\documentclass[11pt]{amsart}

\usepackage{amsmath}
\usepackage{amssymb}
\usepackage{mathrsfs}
\usepackage{comment}
\usepackage{color}
\usepackage[colorlinks,citecolor=blue,urlcolor=black,linkcolor=black]{hyperref}

\newtheorem{theorem}{Theorem}
\newtheorem{claim}[theorem]{Claim}

\newtheorem{lemma}[theorem]{Lemma}

\newtheorem{corollary}[theorem]{Corollary}

\theoremstyle{definition}
\newtheorem{definition}[theorem]{Definition}

\theoremstyle{remark}

\newcount\skewfactor
\def\mathunderaccent#1#2 {\let\theaccent#1\skewfactor#2
\mathpalette\putaccentunder}
\def\putaccentunder#1#2{\oalign{$#1#2$\crcr\hidewidth
\vbox to.2ex{\hbox{$#1\skew\skewfactor\theaccent{}$}\vss}\hidewidth}}

\def\smallbox#1{\leavevmode\thinspace\hbox{\vrule\vtop{\vbox
   {\hrule\kern1pt\hbox{\vphantom{\tt/}\thinspace{\tt#1}\thinspace}}
   \kern1pt\hrule}\vrule}\thinspace}

\newcommand{\cf}{{\rm cf}}

\title{List chromatic numbers and singular compactness}
\author{Shimon Garti}
\address{Einstein Institute of Mathematics,
 The Hebrew University of Jerusalem,
 Jerusalem 9190401, Israel}
\email{shimon.garty@mail.huji.ac.il}

\subjclass[2010]{05C15, 05C63}
\keywords{List chromatic number, bipartite graphs, singular cardinals}

\begin{document}
\let\labeloriginal\label
\let\reforiginal\ref
\def\ref#1{\reforiginal{#1}}
\def\label#1{\labeloriginal{#1}}

\begin{abstract}
Let $G=(V,E)$ be a graph of size $\mu$ where $\mu>\cf(\mu)$.
Assume that $\cf(\mu)\leq\cf(\lambda)\leq\lambda<\mu$ and $\eta^\lambda<\mu$ for every $\eta\in\mu$.
We prove that if ${\rm List}(H)\leq\lambda$ for every $H\leq G$ such that $|H|<\mu$ then ${\rm List}(G)\leq\lambda$.
\end{abstract}

\maketitle

\newpage

\section{Introduction}

An abelain group $A$ is called almost free if every subgroup $B$ of $A$ of size less than the size of $A$ is free.
An interesting problem is whether there is an almost free non-free abelian group or whether almost free implies free.
An important parameter here is the cardinality of $A$.
In many cases this question is independent of \textsf{ZFC}.
This includes the case of a successor of a singular cardinal as proved in \cite{MR1249391}.
But if the cardinality of $A$ is a singular cardinal $\mu$ then almost free abelian groups of size $\mu$ are free.

This phenomenon has been studied by Shelah in \cite{MR389579} and it is called \emph{singular compactness}.
The range of singular compactness is much wider than free abelian groups, and it is known to exist in many mathematical areas including graph theory.
In the general description we have a family of structures, a property $\mathcal{P}$ of these structures and a cardinal $\lambda<\mu$ where $\mu$ is a singular cardinal.
We shall say that the family satisfies singular compactness with respect to $\mathcal{P}$ iff given any $\mathfrak{A}$ of size $\mu$ it is true that $\mathcal{P}(\mathfrak{A})\leq\lambda$ whenever $\mathcal{P}(\mathfrak{B})\leq\lambda$ for every substructure $\mathfrak{B}$ of $\mathfrak{A}$ of size less than $\mu$.

In this paper we deal with singular compactness applied to graph theory.
More specifically, we prove a statement about the list chromatic number in infinite graphs.
Let $G=(V,E)$ be a graph.
A function $f:V\rightarrow{\rm Ord}$ is a \emph{good coloring} iff $f(x)\neq f(y)$ whenever $\{x,y\}\in E$.
If $\kappa$ is a cardinal then a $\kappa$-assignment for $G$ is a function $F:V\rightarrow[{\rm Ord}]^\kappa$.
The list-chromatic number of $G$, denoted ${\rm List}(G)$, is the minimal $\kappa$ such that for every $\kappa$-assignment $F:V^G\rightarrow[{\rm Ord}]^\kappa$ one can find a good coloring $f:V^G\rightarrow{\rm Ord}$ so that $f(x)\in F(x)$ whenever $x\in V^G$.

The list chromatic number was defined, independently, by Vizing in \cite{MR498216} and by Erd\H{o}s, Rubin and Taylor in \cite{MR593902}.
Singular compactness with respect to the list chromatic number is the following statement.
Suppose that $G=(V,E)$ is a graph where $|V|=\mu$ and $\mu>\cf(\mu)$.
If $\lambda<\mu$ and ${\rm List}(H)\leq\lambda$ for every $H\leq G$ of size less than $\mu$ then ${\rm List}(G)\leq\lambda$ as well.
It is known that the list chromatic number fails to satisfy singular compcatness without additional assumptions, and a simple example for this failure is given by Komj\'{a}th in \cite{MR3096584}.
Komj\'{a}th proved that if $G$ is a bipartite graph with $V^G=A\cup B, |A|=\aleph_0$ and $|B|<2^{\aleph_0}$ then ${\rm List}(G)\leq\aleph_0$.
On the other hand, there are many bipartite graphs $G$ such that $V^G=A\cup B, |A|=\aleph_0$ and $|B|=2^{\aleph_0}$ and ${\rm List}(G)>\aleph_0$.
Hence if one forces $2^{\aleph_0}=\mu$ where $\mu>\cf(\mu)$ then singular compactness fails with respect to ${\rm List}(G)$.

The example of Komj\'{a}th generalizes to larger cardinalities, and one concludes that ${\rm List}(G)$ does not satisfy singular compactness without further assumptions.
However, Usuba proved in \cite{MR4732367} that under some cardinal arithmetic assumptions on the singular cardinal $\mu$ and the power-set operation below $\mu$ one can prove singular compactness for ${\rm List}(G)$ where $|G|=\mu$.
He asked whether the assumption $\eta^\lambda<\mu$ for every $\eta\in\mu$ will be sufficient, and gave some partial results, see \cite[Question 3.8]{MR4732367}.
Our goal is to give a positive answer to this question.

\newpage

\section{Singular compactness}

In this section we give sufficient conditions for singular compactness with respect to the list chromatic number in terms of cardinal arithmetic.
As a first step we introduce a slight improvement upon \cite[Fact 2.4]{MR4732367} by proving a similar statement but weakening one of the assumptions.
Given a graph $G=(V,E)$ and $x\in V$ let $E^x=\{y\in V:\{x,y\}\in E\}$.

\begin{lemma}
\label{lem2.4} Let $G=(V,E)$ be a graph and suppose that $\nu\leq\tau$ are infinite cardinals.
Assume that there are $W_0,W_1\subseteq V$ such that $|W_0|=\tau,|W_1|\geq\tau^\nu,W_0\cap W_1=\varnothing$ and $|E^x\cap W_0|\geq\nu$ for every $x\in W_1$.
Then ${\rm List}(G)>\nu$.
\end{lemma}

\begin{proof}
Since $H\leq G$ implies ${\rm List}(H)\leq{\rm List}(G)$ it is sufficient to find some $H\leq G$ for which ${\rm List}(H)>\nu$.
Hence we assume that $|E^x\cap W_0|=\nu$ for every $x\in W_1$ by eliminating some edges from $E$ if needed.

Fix a disjoint family $\{A_y:y\in W_0\}\subseteq[\tau]^\nu$.
The cardinality of the product $\prod_{y\in W_0}A_y$ is $\tau^\nu$.
Let $H$ be the graph induced by $W_0\cup W_1$.
We shall define a $\nu$-assignment for $H$ so that no function $f:V^H\rightarrow{\rm Ord}$, amenable to this assignment, will be good.
This will show that ${\rm List}(H)>\nu$ and hence ${\rm List}(G)>\nu$ as desired.

Enumerate the elements of $\prod_{y\in W_0}A_y$ by $(g_z:z\in W_1)$.
If $|W_1|>\tau^\nu$ then either take a subset of $W_1$ or size $\tau^\nu$ or use repetitions.
For every $y\in W_0$ let $F(y)=A_y$ and for every $z\in W_1$ let $F(z)=\{g_z(y):y\in W_0,\{y,z\}\in E\}$.
Observe that $F$ is a $\nu$-assignment since $|E^x\cap W_0|=\nu$ for every $x\in W_1$ and since $A_y\cap A_z=\varnothing$ whenever $y,z\in W_0$ and $y\neq z$.

Suppose that $f$ is a coloring of $H$ for which $f(x)\in F(x)$ whenever $x\in W_0\cup W_1=V^H$.
Notice that $f\upharpoonright W_0\in\prod_{y\in W_0}A_y$ and choose $z\in W_1$ such that $f\upharpoonright W_0=g_z$.
By definition, $F(z)=\{g_z(y):y\in E^z\cap W_0\}=\{f(y):y\in E^z\cap W_0\}$.
But $f(z)\in F(z)$ so $f(z)=f(y)$ for some $y\in E^z\cap W_0$.
This means that $f$ is not a good coloring, so we are done.
\end{proof}

The next lemma that we need will be a variation of a parallel statement from \cite{MR4732367}.
Recall that if $A$ is a set and $\delta$ is an ordinal then a $\delta$-filtration of $A$ is a $\subseteq$-increasing and continuous sequence $(A_\alpha:\alpha\in\delta)$ such that $|A_\alpha|<|A|$ for every $\alpha\in\delta$ and $A=\bigcup_{\alpha\in\delta}A_\alpha$.
We need a weaker version of this concept.

\begin{definition}
  \label{defweakfiltration} Let $V$ be a set and let $\delta$ be an ordinal.
  A sequence of sets $(A_\alpha:\alpha\in\delta)$ is a $\delta$-weak-filtration of $V$ when $|A_\alpha|<|V|$ for every $\alpha\in\delta$, $A_\alpha\subseteq{A_\beta}$ whenever $\alpha\in\beta$ and $V=\bigcup_{\alpha\in\delta}A_\alpha$.
\end{definition}

The difference between filtration and weak filtration is that in the latter we omit the requirement of continuity from the sequence of sets that constitutes the filtration.
The following is an adaptation of \cite[Lemma 3.5]{MR4732367} to the concept of weak filtration.

\begin{lemma}
\label{lem3.5} Let $G=(V,E)$ be a graph and let $\nu$ be an infinite cardinal.
Let $(A_\alpha:\alpha\in\delta)$ be a weak filtration of $V$ such that:
\begin{enumerate}
\item [$(\aleph)$] ${\rm List}(A_\alpha)\leq\nu$ for every $\alpha\in\delta$.
\item [$(\beth)$] $|E^x\cap A_\alpha|<\nu$ for every $\alpha\in\delta$ and every $x\in V-A_\alpha$.
\item [$(\gimel)$] $\delta\leq\cf(\nu)$.
\end{enumerate}
Then ${\rm List}(G)\leq\nu$.
\end{lemma}

\begin{proof}
  Let $F:V\rightarrow[{\rm Ord}]^\nu$ be a $\nu$-assignment.
  Our goal is to construct a good coloring $f$ of $V$ so that $f(x)\in F(x)$ for every $x\in{V}$.
  We proceed by induction on $\alpha\in\delta$.
  Arriving at $\alpha$, suppose that $f\upharpoonright{A_\beta}$ is a good coloring that satisfies $f(x)\in F(x)$ for each $x\in{A_\beta}$, and this holds for every $\beta\in\alpha$.

  Define $f\upharpoonright\bigcup_{\beta\in\alpha}A_\beta = \bigcup_{\beta\in\alpha}f\upharpoonright{A_\beta}$.
  Let $Y=A_\alpha-\bigcup_{\beta\in\alpha}A_\beta$, and notice that $Y$ can be empty.
  Since $Y\subseteq{A_\alpha}$ we see that ${\rm List}(Y)\leq{\rm List}(A_\alpha)\leq\nu$.
  If $x\in{Y}$ then $|E^x\cap\bigcup_{\beta\in\alpha}A_\beta|<\nu$ since $\delta\leq\cf(\nu)$ and $\alpha\in\delta$.
  Thus one can define $F'(x)=F(x)-(f''(E^x\cap\bigcup_{\beta\in\alpha}A_\beta))$, so $|F'(x)|=\nu$ and hence $F'$ is a $\nu$-assignment for $Y$.

  Let $g$ be a good coloring of $Y$ such that $g(x)\in F'(x)$ for every $x\in{Y}$.
  Define $f\upharpoonright{A_\alpha} = (f\upharpoonright\bigcup_{\beta\in\alpha}A_\beta)\cup{g}$, and let $f=\bigcup_{\alpha\in\delta}f\upharpoonright{A_\alpha}$.
  It follows that $f$ is a good coloring of $V$, and $f(x)\in F(x)$ whenever $x\in{V}$, so we are done.
\end{proof}

We can prove now our main result in a similar way to that of \cite[Proposition 3.6]{MR4732367}, but using Lemma \ref{lem2.4} we can relax the assumptions needed for singular compactness.
Notice that if our singular cardinal is strong limit then singular compactness holds at $\mu$ for the list chromatic number with no additional parameters.
That is, for every $\lambda<\mu$ if ${\rm List}(H)\leq\lambda$ whenever $H\leq G$ is of size less than $\mu$ then ${\rm List}(G)\leq\lambda$ as well.

\begin{theorem}
\label{thmmt} Assume that:
\begin{enumerate}
\item [$(a)$] $\mu>\cf(\mu)=\theta$ and $\theta\leq\cf(\lambda)\leq\lambda<\mu$.
\item [$(b)$] $\eta^\lambda<\mu$ for every $\eta<\mu$.
\item [$(c)$] $G=(V,E)$ is a graph of size $\mu$.
\end{enumerate}
If ${\rm List}(H)\leq\lambda$ for every $H\leq G$ of size less than $\mu$ then ${\rm List}(G)\leq\lambda$.
Hence if $\mu$ is strong limit then the statement holds for every $\lambda\in{\rm Reg}\cap[\theta,\mu)$.
\end{theorem}

\begin{proof}
Let $(\mu_i:i\in\theta)$ be an increasing continuous sequence of cardinals such that $\theta,\lambda<\mu_0$ and $\mu=\bigcup_{i\in\theta}\mu_i$.
For every $i\in\theta$ let $\nu_i=\mu_i^\lambda$, so $\nu_i<\mu$ for every $i\in\theta$ by virtue of $(b)$. Notice that $\nu_i^\lambda=\nu_i$ for every $i\in\theta$.

The following observation is central to the proof.
If $U\subseteq V$ and $|U|=\nu_i$ then $|Y_U|<\nu_i$ where $Y_U=\{x\in V:|E^x\cap U|\geq\lambda\}$.
Indeed, if $U$ forms a counterexample then one can apply Lemma \ref{lem2.4} to the subgraph of $G$ induced by $U\cup Y_U$ and conclude that ${\rm List}(U\cup Y_U)>\lambda$ contrary to the assumption that ${\rm List}(H)\leq\lambda$ whenever $H\leq G$ and $|H|<\mu$.

Bearing in mind the above observation we choose an array $(M^\alpha_i:i\in\theta,\alpha\in\lambda^+)$ of elementary submodels of $\mathcal{H}(\chi)$ where $\chi$ is a sufficiently large regular cardinal, such that the following requirements are met:
\begin{enumerate}
\item [$(\alpha)$] $|M^\alpha_i|=\nu_i$ and $\nu_i+1\subseteq M^\alpha_i$ for every $i\in\theta,\alpha\in\lambda^+$.
\item [$(\beta)$] $(M^\alpha_i:i\in\theta)$ is $\subseteq$-increasing\footnote{Notice that we do not require this sequence to be continuous.} for every $\alpha\in\lambda^+$.
\item [$(\gamma)$] If $\alpha\in\lambda^+$ is a limit ordinal then $M^\alpha_i=\bigcup_{\beta\in\alpha}M^\beta_i$ for every $i\in\theta$.
\item [$(\delta)$] $(M^\alpha_i:i\in\theta)\in M^{\alpha+1}_0$ for every $\alpha\in\lambda^+$.
\item [$(\varepsilon)$] $\mu,G\in M^\alpha_i$ for every $i\in\theta,\alpha\in\lambda^+$.
\end{enumerate}
We may assume, without loss of generality, that $V=\mu$ so $G=(\mu,E)$.
For every $i\in\theta$ let $M_i=\bigcup_{\alpha\in\lambda^+}M^\alpha_i$ and let $A_i=M_i\cap\mu$.
Notice that $M^\alpha_i\in M_i$ for every $i\in\theta,\alpha\in\lambda^+$ and $\mu\subseteq\bigcup_{i\in\theta}M_i$.
Also, $|A_i|=\nu_i<\mu$ and hence $(A_i:i\in\theta)$ is a weak filtration of $V$.

We claim that for every $i\in\theta$ and every $x\in V-A_i$ it is true that $|E^x\cap A_i|<\lambda$.
If we prove this claim then from Lemma \ref{lem3.5} we will infer that ${\rm List}(G)\leq\lambda$ and the proof of the theorem will be accomplished.
Assume, therefore, that the claim fails and fix $i\in\theta$ and $x\in\mu-A_i$ such that $|E^x\cap A_i|\geq\lambda$.
Since $A_i=M_i\cap\mu=\bigcup_{\alpha\in\lambda^+}M^\alpha_i\cap\mu$ one can choose $\alpha\in\lambda^+$ for which $|E^x\cap M^\alpha_i|\geq\lambda$.

Define $Y=\{y\in\mu:|E^y\cap M^\alpha_i|\geq\lambda\}$.
The set $Y$ is definable in $M_i$ since $M^\alpha_i\in M_i$ as well as the other parameters, so $Y\in M_i$.
By the observation at the beginning of the proof we have $|Y|<\nu_i$ and hence $Y\subseteq M_i$.
But $x$ satisfies $|E^x\cap M^\alpha_i|\geq\lambda$ and hence $x\in Y\subseteq M_i$.
This means that $x\in M_i\cap\mu=A_i$, which is impossible as $x\in\mu-A_i$, so we are done.
\end{proof}

We indicate that if $\mu$ is a singular cardinal of countable cofinality, that is $\theta=\aleph_0$, then one can require that the sequence $(M^\alpha_i:i\in\theta)$ will be continuous for every $\alpha\in\lambda^+$.
Indeed, the sequence of cardinals $(\nu_i:i\in\omega)$ will be continuous in this case (simply because there are no limit points in the sequence).
This means that in Lemma \ref{lem3.5} one can use filtrations rather than weak-filtrations.

As mentioned in the introduction, the invariant of list chromatic numbers does not satisfy singular compactness without further assumptions.
One may wonder whether $\eta^\lambda<\mu$ for every $\eta<\mu$ is the optimal assumption for singular compactness.
There is no much room for improving this assumption since $2^\omega=\mu$ already implies the failure of singular compactness with respect to the list chromatic number at $\mu$, and $\omega$ can be replaced by every strongly inaccessible cardinal $\kappa$.

Pondering upon the possibility that $2^\lambda<\mu$ is insufficient for singular compactness at $\mu$ (with $\lambda$ as a parameter) we conclude from the above theorem that in such a case there must be $\eta<\mu$ so that $\eta^\lambda\geq\mu$.
The following claim shows that this cardinal $\eta$ has some interesting properties.

\begin{claim}
\label{clmoptimal} Suppose that $2^\lambda<\mu, \mu$ is a singular cardinal and singular compactness for the list chromatic number fails at $\mu$ with respect to $\lambda$.
Then there is a singular cardinal $\eta<\mu$ which satisfies $\cf(\eta)\leq\lambda, \eta^\lambda\geq\mu$ and $\theta<\eta\Rightarrow\theta^\lambda<\eta$.
\end{claim}

\begin{proof}
Let $\eta$ be the first cardinal below $\mu$ for which $\eta^\lambda\geq\mu$.
Such a cardinal exists by virtue of Theorem \ref{thmmt} and the assumptions of the claim.
If $\theta<\eta$ and $\theta^\lambda\geq\eta$ then $\theta^\lambda=(\theta^\lambda)^\lambda\geq\eta^\lambda\geq\mu$, and this is impossible since $\eta$ is the first cardinal which satisfies $\eta^\lambda\geq\mu$, hence $\theta<\eta\Rightarrow\theta^\lambda<\eta$.
Thus necessarily $\cf(\eta)\leq\lambda$.
Indeed, if $\cf(\eta)>\lambda$ then $\eta^\lambda=\bigcup_{\theta<\eta}\theta^\lambda=\eta<\mu$.
The proof of the claim is accomplished.
\end{proof}

There are some interesting upshots which follow from this claim.
Call $\mu$ a \emph{weird} cardinal iff $\mu$ is singular and there is some $\lambda<\mu$ such that $2^\lambda\geq\mu$ and singular compactness for list chromatic numbers fails at $\mu$ with respect to $\lambda$.
If $\lambda$ is specified then we shall say that $\mu$ is $\lambda$-weird.
Let us show that in some sense the collection of wierdos is small.
Fix a cardinal $\kappa$ of uncountable cofinality.
Suppose that $\mu<\kappa$ is $\lambda$-weird for some $\lambda<\mu$.
From Claim \ref{clmoptimal} we see that there is some $\eta<\mu$ so that $\eta^\lambda\geq\mu$.
If $\eta^\lambda<\kappa$ for every $\lambda$-wierd cardinal $\mu<\kappa$ then $\kappa$ will be called a \emph{sane} cardinal.

\begin{corollary}
\label{corweird} Let $\lambda$ be an infinite cardinal.
Suppose that $\kappa$ is sane.
Then the set of weird cardinals below $\kappa$ is not stationary in $\kappa$.
\end{corollary}

\begin{proof}
As a first step, we fix an infinite cardinal $\lambda$ and show that the set of $\lambda$-weird cardinals below $\kappa$ is not stationary.
Call this set $S$, and assume towards contradiction that $S$ is stationary in $\kappa$.
By cutting-off an initial segment we may assume that $2^\lambda<\min(S)$.
Applying Claim \ref{clmoptimal} to each element of $S$, let $\eta_\mu<\mu$ be the pertinent cardinal, and since $\kappa$ is sane we see that $\eta_\mu^\lambda<\kappa$ for every $\mu\in S$.

The function $h(\mu)=\eta_\mu$ is regressive on $S$, hence there is a stationary subset $T\subseteq S$ and a fixed $\eta$ such that $\mu\in T\Rightarrow\eta_\mu=\eta$.
Since $\kappa$ is sane we see that $\eta^\lambda<\kappa$.
On the other hand, since $T$ is unbounded in $\kappa$ one can pick some $\mu\in T$ so that $\eta^\lambda<\mu$ and this is impossible since $\eta^\lambda=\eta_\mu^\lambda\geq\mu$.

For the second step, let $S_0$ be the set of weird cardinals below $\kappa$.
For every $\mu\in S_0$ there exists some $\lambda_\mu<\mu$ such that $\mu$ is $\lambda_\mu$-weird.
If $S_0$ is stationary then there are a fixed $\lambda$ and a stationary subset $S_1$ of $S_0$ so that $\mu\in S_1\Rightarrow\lambda_\mu=\lambda$.
Apply the first step to $S_1$ in order to obtain the desired contradiction.
\end{proof}

\newpage

\section{Reflections on reflection}

In this section we deal with another phenomenon related to the list chromatic number.
One can analyze reflection principles with respect to invariants of graphs, including the list chromatic number.
The general idea is to show that if a graph has a certain value of the list chromatic number then there is already a \emph{small} graph with the same property.

\begin{definition}
\label{defrplist} Reflection for list chromaticity. \newline
Assume that $\lambda\geq\aleph_2$.
\begin{enumerate}
\item [$(\aleph)$] ${\rm RP}({\rm List},\lambda)$ is the statement that if $G$ is a graph, $|G|\leq\lambda$ and ${\rm List}(G)>\aleph_0$ then there exists $H\leq G$ such that $|H|=\aleph_1$ and ${\rm List}(H)>\aleph_0$.
\item [$(\beth)$] ${\rm RP}({\rm List})$ is the statement that ${\rm RP}({\rm List},\lambda)$ holds for every $\lambda\geq\aleph_2$.
\item [$(\gimel)$] Global ${\rm RP}({\rm List})$, or ${\rm GRP}({\rm List})$, is the statement that for every $\lambda\geq\aleph_2$ and every graph $G$ such that $|G|\leq\lambda$ if ${\rm List}(G)>\kappa$ then there exists $H\leq G$ such that $|H|=\kappa^+$ and ${\rm List}(H)>\kappa$.
\end{enumerate}
\end{definition}

Usuba showed that ${\rm RP}({\rm List})$ has some consistency strength by proving that if ${\rm RP}({\rm List})$ holds then for every $\lambda\geq\cf(\lambda)>\omega$ either $2^\lambda>\lambda^+$ or every stationary subset of $S^{\lambda^+}_\omega$ reflects.
Consequently, either $\square_\lambda$ fails or $\textsf{SCH}$ at $\lambda$ fails under ${\rm RP}({\rm List})$, see \cite[Proposition 4.6]{MR4732367}.
Another interesting fact is that ${\rm RP}({\rm List},2^\omega)$ implies $2^\omega=\omega_1$.
Based on these theorems, Usuba asked whether ${\rm RP}(\rm List)$ implies \textsf{SCH}.
Our next goal is to prove that ${\rm GRP}({\rm List})$ implies \textsf{SCH}.
In fact, ${\rm GRP}({\rm List})$ implies a bit more and this will be discussed after the proof.
The proof itself is a generalization of Komj\'{a}th in \cite{MR3096584} who proved similar statements in the countable case.

\begin{theorem}
\label{thmgrplist} ${\rm GRP}({\rm List})$ implies SCH.
\end{theorem}

\begin{proof}
Suppose that $\mu$ is a strong limit cardinal, either regular or singular.
Let $G$ be a bipartite graph with a bipartition of $V^G$ into $A$ and $B$ such that $|A|=\mu$ and $|B|=\lambda<2^\mu$.
Let us prove that ${\rm List}(G)\leq\mu$.

Firstly, let $A=\{v_\alpha:\alpha\in\mu\}$ and $B=\{w_\beta:\beta\in\lambda\}$.
Secondly, let $F$ be any $\mu$-assignment defined on $A\cup B$.
By induction on $\alpha\in\mu$ we choose $\eta_f\in F(v_\alpha)$ for every function $f\in{}^\alpha 2$ such that $f\neq g\Rightarrow\eta_f\neq\eta_g$ whenever $\beta\leq\alpha$ and $g\in{}^\beta 2$.
The choice is possible since $\mu$ is strong limit.
Specifically, if $\alpha\in\mu$ let $\theta=\bigcup\{2^\beta:\beta\leq\alpha\}$ so $\theta<\mu$.
Since $|F(v_\alpha)|=\mu$ one can choose a distinct element $\eta_f\in F(v_\alpha)$ for every $f\in\bigcup\{{}^\beta 2:\beta\leq\alpha\}$.

Suppose that $h\in{}^\mu 2$ and let $T_h=\{\eta_{h\upharpoonright\alpha}:\alpha\in\mu\}$.
By our choice if $\alpha\neq\beta$ then $\eta_{h\upharpoonright\alpha}\neq\eta_{h\upharpoonright\beta}$ and hence $|T_h|=\mu$.
On the other hand if $g\neq h$ then $|T_g\cap T_h|<\mu$ since $g\upharpoonright\beta\neq h\upharpoonright\beta$ from some $\beta$ onwards.
It follows that if $\beta\in\lambda$ then there is at most one function $h\in{}^\mu 2$ for which $F(w_\beta)\subseteq T_h$.
Since $\lambda<2^\mu$ one can choose $h\in{}^\mu 2$ such that $\forall\beta\in\lambda,\neg(F(w_\beta)\subseteq T_h)$.
Define $f(v_\alpha)=\eta_{h\upharpoonright\alpha}$ for every $\alpha\in\mu$ and choose $f(w_\beta)\in F(w_\beta)-T_h$ for every $\beta\in\lambda$.
By the above considerations, $f$ is a good coloring of $G$ and hence ${\rm List}(G)\leq\mu$.

Consider now the complete bipartite graph $K_{\mu,2^\mu}$.
Let $A\cup B$ be a bipartition of this graph with $A=\{a_\alpha:\alpha\in\mu\}$ and $B=\{b_\beta:\beta\in 2^\mu\}$.
Let us show that ${\rm List}(K_{\mu,2^\mu})>\mu$.

As a first step choose $\mathcal{A}=(A_\alpha:\alpha\in\mu)$ such that $|A_\alpha|=\mu$ for every $\alpha\in\mu$ and $\alpha_0\neq\alpha_1\Rightarrow A_{\alpha_0}\cap A_{\alpha_1}=\varnothing$.
Now let $E=\prod\{A_\alpha:\alpha\in\mu\}$, so $|E|=\mu^\mu=2^\mu$.
Define a $\mu$-assignment $F:A\cup B\rightarrow[{\rm Ord}]^\mu$ as follows.
For every $\alpha\in\mu$ let $F(a_\alpha)=A_\alpha$.
Fix a bijection $h:B\rightarrow E$, say $h(b_\beta)=g_\beta\in E$ for every $\beta\in 2^\mu$, and define $F(b_\beta)=\{g_\beta(\alpha):\alpha\in\mu\}$.
Notice that $|F(b_\beta)|=\mu$ by the disjointness of the $A_\alpha$s.

Suppose that $f$ is a coloring defined on $A\cup B$ so that $f(v)\in F(v)$ for every $v\in A\cup B$.
Since $f\upharpoonright{A}\in E$ there exists a unique $\beta\in 2^\mu$ such that $f\upharpoonright{A}=g_\beta$.
It follows that $f(b_\beta)\in F(b_\beta)=\{g_\beta(\alpha):\alpha\in\mu\}=\{f(a_\alpha):\alpha\in\mu\}$, so for some $a_\alpha\in A$ we have $f(b_\beta)=f(a_\alpha)$.
This fact shows that $f$ is not a good coloring since every element of $A$ is connected with every element of $B$, hence ${\rm List}(K_{\mu,2^\mu})>\mu$ as desired.

Assume now that ${\rm GRP}({\rm List})$ holds and $\mu$ is a strong limit cardinal.
If $2^\mu>\mu^+$ then $G=K_{\mu,2^\mu}$ is a graph of size greater than $\mu^+$ and ${\rm List}(G)>\mu$.
But if $H\leq G$ and $|H|=\mu^+$ then $H$ is a subgraph of $K_{\mu,\mu^+}$ and hence ${\rm List}(H)\leq\mu$ contradicting ${\rm GRP}({\rm List})$.
We conclude, therefore, that $2^\mu=\mu^+$ as required.
\end{proof}

By the fundamental work of Gitik in \cite{MR1007865} we know that the consistency strength of ${\rm GRP}({\rm List})$ is at least a measurable cardinal $\kappa$ with $o(\kappa)=\kappa^{++}$.
We indicate, however, that ${\rm GRP}({\rm List})$ implies more than \textsf{SCH}, as it implies $2^\kappa=\kappa^+$ at every strongly inaccessible cardinal $\kappa$.

\newpage

\bibliographystyle{alpha}
\bibliography{arlist}

\end{document}